\documentclass[12pt,a4paper,twoside]{amsart}
\usepackage{geometry}
\usepackage{txfonts,pxfonts,tikz}
\usepackage[T1]{fontenc}
\usepackage[utf8]{inputenc}

\usepackage{amsmath}
\usepackage{amssymb}
\usepackage{amsthm}

\usepackage{mathtools}
\usepackage{stmaryrd}

\usepackage{wrapfig}
\usepackage{subcaption}

\usepackage{enumitem}
\usepackage{verbatim}
\usepackage{makecell}
\usepackage{xcolor}

\usepackage{hyperref}
\hypersetup{colorlinks,linkcolor={blue!75!black},citecolor={blue!75!black},urlcolor={blue!75!black}}

\usepackage{setspace}
\usepackage{microtype} 

\DeclareMathOperator*{\supp}{supp}

\newcommand{\norm}[1]{\|#1\|}

\newcommand{\R}{\mathbb{R}}
\newcommand{\Z}{\mathbb{Z}}

\newcommand{\abs}[1]{\left\vert#1\right\vert}
\newcommand{\dint}{\,\mathrm{d}}

\newcommand{\hP}{h_{P}}
\newcommand{\nuP}{\nu_{P}}

\mathtoolsset{showonlyrefs}

\theoremstyle{theorem}
\newtheorem{theorem}{Theorem}[section]
\newtheorem{corollary}[theorem]{Corollary}
\newtheorem{proposition}[theorem]{Proposition}
\newtheorem{lemma}[theorem]{Lemma}

\theoremstyle{definition}

\theoremstyle{remark}
\newtheorem{remark}[theorem]{Remark}

\numberwithin{equation}{section}

\allowdisplaybreaks[4]

\title[Newton-height decay estimates for oscillatory integral operators]{Sharp decay estimates for $(2+1)$-dimensional oscillatory integral operators via Newton height
}
\author{Shaozhen Xu}

\address{School of Information Engineering, Nanjing Xiaozhuang University, Nanjing 211171, China}
\email{shaozhen@nju.edu.cn}

\begin{document}

\begin{abstract}
We study $(2+1)$-dimensional oscillatory integral operators of the form
\[
T_\lambda f(x,y)=\int_{\R}e^{i\lambda P(x,y)t^k}\psi(x,y,t)f(t)\dint t,\qquad k\ge1,
\]
where the phase $P$ is a real-analytic function with a critical point at the origin. We establish the sharp $L^2\to L^2$ decay rate of $\frac12\min\{1/\hP, 1/k\}$, where $\hP$ denotes Varchenko's Newton height of $P$. The two terms in the minimum reflect a natural competition between the spatial degeneracy of $P$ and the temporal degeneracy of $t^k$; their optimality is confirmed by a Knapp-type and a focusing example, respectively. A $TT^{*}$ reduction transforms the $L^2$ estimate into a scalar oscillatory integral, allowing Varchenko's theorem to apply directly. Building on this foundation, complex interpolation yields the sharp $L^2\to L^{2k+2}$ bound. Finally, in the regime $\hP\ge k$, we obtain sharp $L^2\to L^p$ decay estimates for all $p$.
\end{abstract}

\maketitle

\setcounter{tocdepth}{1}
\tableofcontents

\section{Introduction}

 It's a classical problem in harmonic analysis to study the asymptotic decay of the scalar oscillatory integral
\begin{equation}
	I(\lambda)=\int_{\R^n}e^{i\lambda Q(x)}\phi(x)\dint x
\end{equation}
as the real parameter $\lambda \to \infty$. If $Q$ is nondegenerate, the stationary phase method gives a complete asymptotic expansion. If $Q$ is degenerate but real-analytic, the decay rate is governed by the geometry of $Q$. Varchenko \cite{Var76} proved that, for generic real-analytic phases in any dimension, the sharp rate is determined by the \emph{Newton distance} of $Q$. In two dimensions no genericity assumption is needed, provided one works in adapted coordinates: the sharp rate is $-1/h_Q$, where the Newton height $h_Q$ is the coordinate at which the bisectrix meets the boundary of the Newton polygon. Specifically,
\begin{equation}\label{Var-scalar}
	\abs{I(\lambda)}\le C\,(1+\abs{\lambda})^{-1/h_Q}\bigl(\log(2+\abs{\lambda})\bigr)^{\nu_Q},\qquad \nu_Q\in\{0,1\},
\end{equation}
where the logarithmic factor is present exactly when the bisectrix meets the boundary at a vertex of the Newton polygon in some adapted coordinate system; in particular $\nu_Q=0$ when $h_Q=1$. In this paper, we prove an operator analogue of the scalar bound \eqref{Var-scalar}.

The operator analogue of $I(\lambda)$ is
\begin{equation}
	\mathcal{T}_\lambda f(x)=\int_{\R^{n_Y}}e^{i\lambda S(x,y)}\psi(x,y)f(y)\dint y,\quad x\in\R^{n_X}.
\end{equation}
This is an $(n_X+n_Y)$-\emph{dimensional} oscillatory integral operator. The main goal is to establish the sharp $L^q\to L^p$ decay of $\mathcal{T}_\lambda$. If $n_X=n_Y=n$ and $S$ is nondegenerate, H\"ormander \cite{Hor73} proved the sharp $L^2\to L^2$ decay rate is $-n/2$. In the degenerate case, one hopes to describe the decay by the Newton polyhedron of $S$, just as Varchenko did for $I(\lambda)$. The $(1+1)$-dimensional theory is well developed; see \cite{PhoSte92,See93,PhoSte94,PhoSte97,PhoSte98,PhoSteStu01,Ryc01,Gre04,Yan04,Gre05,Xia17,SXY19}. Beyond it, little is known. Higher-dimensional phases are harder and need new tools. For more background, see the introduction of \cite{Xu23}.

\subsection*{The operators, and what is new}
Recent work on $(2+1)$-dimensional operators has focused on the monomial-sum phase model
\begin{equation}\label{old-O.I.O}
	\widetilde{T}_\lambda f(x,y)=\int_{\R}e^{i\lambda\left(x^mt^k+y^nt^l\right)}\psi(x,y,t)f(t)\dint t,\qquad (k,l,m,n)\in(\Z^{+})^4.
\end{equation}
Here $\psi$ is supported in the unit ball. For $m=1, k=2, n=2, l=1$, sharp $L^4\to L^4$ bounds were proved in \cite{Xu23} using the broad--narrow method. Later, \cite{TX24} used fractional integration to prove the optimal $L^2\to L^6$ estimate. For general exponents with $k>l$, Stein's complex interpolation was used in \cite{Xu25} to prove
\begin{equation}\label{Xu25-main}
	\norm{\widetilde{T}_\lambda f}_{L^{2k+2}(\R^2)}\le C\lambda^{-\frac{1}{2(k+1)}\left(\frac{1}{m}+\frac{1}{\max\{n,l\}}\right)}\norm{f}_{L^2(\R)}.
\end{equation}
This is sharp if $l\le n$. The balanced case $k=l$ was left open (see \cite[Remark~1.2]{Xu25}). In this case, the two monomial terms merge into a single phase $(x^m+y^n)t^k$. This merging is exactly what breaks the earlier method. The interpolation relies on a sharp $L^2$ endpoint estimate. Earlier works obtained this by iterating one-dimensional Phong--Stein bounds, which becomes impossible once the phase merges.

In this paper, we study the operators
\begin{equation}\label{Key-O.I.O}
	T_\lambda f(x,y)=\int_{\R}e^{i\lambda P(x,y)t^k}\psi(x,y,t)f(t)\dint t,\qquad k\in\Z^{+}.
\end{equation}
Here the \emph{spatial phase} $P(x,y)$ can be any real-analytic function satisfying the following \textbf{Standing assumptions}. This includes the balanced case of \eqref{old-O.I.O}, where $P(x,y)=x^m+y^n$. We no longer rely on the monomial-sum structure. Our method has two main new features:

\emph{(a) The decay depends on the Newton height of $P$.} By a $TT^{*}$ reduction, the kernel of $T_\lambda^{*}T_\lambda$ takes the form
\begin{equation}\label{intro-kernel}
	K(t,s)=\int_{\R^2}e^{i\lambda P(x,y)\left(t^k-s^k\right)}\,\overline{\psi(x,y,s)}\,\psi(x,y,t)\dint x\dint y.
\end{equation}
The variable $W=t^k-s^k$ is constant during the $(x,y)$ integration. Thus \eqref{intro-kernel} is a scalar oscillatory integral with phase $P$ and large parameter $\lambda W$. We apply Varchenko's estimate \eqref{Var-scalar} directly. Using Schur's test, this gives the sharp $L^2\to L^2$ decay in terms of the Newton height $\hP$ (Theorem~\ref{Thm-L2}). So a scalar invariant of \cite{Var76} controls the operator norm. The earlier monomial-sum results are the special case $\hP=\bigl(\tfrac1m+\tfrac1n\bigr)^{-1}$.

\emph{(b) Complex interpolation works for general $P$.} In \cite{Xu25}, the $L^2$ endpoint of the analytic family used the substitution $w=s+t^k$. This separated the phase into $x^m w+y^n t^l$, and one could iterate one-dimensional Phong--Stein estimates \cite{PhoSte94}. For a general $P$, this substitution gives the phase $P(x,y)w$. The variables $x$ and $y$ do not separate, so we cannot iterate. Following the method of Stein \cite{Ste93}, we rewrite $T_K^0$ as the composition of two operators similar to $T_\lambda$ for which we apply our $L^2$ estimates twice to obtain the desired results. So the interpolation works for any $P$. It also gives the sharp exponent $1/\hP$ at the endpoint, with no $1/k$ balancing.

\subsection*{Standing assumptions}
We assume $\psi\in C_c^{\infty}(B^3(0,1))$ is supported in a small neighbourhood of the origin. Moreover, let $P$ be real-analytic on this neighbourhood satisfying 
\[P(0,0)=0,\qquad \nabla P(0,0)=\mathbf 0.\]
In coordinates adapted to $P$ (see \cite{Var76}), we write
\begin{equation}\label{def-h}
	\hP:=\text{Newton distance of }P,\qquad \nuP:=\text{its logarithmic multiplicity }\in\{0,1\}.
\end{equation}
Observe that the assumptions $P(0,0)=0$ and $\nabla P(0,0)=\mathbf 0$ imply that $h_P\ge1$. Furthermore, as noted above, one has $\nu_P=0$ whenever $h_P=1$. Constants $C$ may depend on $\psi,P,k$ but not on $\lambda$. We write $A\lesssim B$ to mean $A\le CB$, and $A\approx B$ to mean $A\lesssim B\lesssim A$.

\subsection*{Main results}
Our first result is the sharp $L^2\to L^2$ estimate, obtained via a $TT^{*}$ reduction, Varchenko's scalar bound for the kernel \eqref{intro-kernel}, and Schur's test. We define the logarithmic exponent
\begin{equation}\label{def-M}
	M_{P,k}:=\begin{cases}
		\nuP/2, & 1/\hP<1/k,\\
		(\nuP+1)/2, & 1/\hP=1/k,\\
		0, & 1/\hP>1/k.
	\end{cases}
\end{equation}

In particular, since $\hP\ge 1$, the $k=1$ case simplifies to

\begin{equation}\label{def-M1}
	M_{P,1}=\begin{cases}
		\nuP/2, & \hP>1,\\
		1/2, & \hP=1.
	\end{cases}
\end{equation}

\begin{theorem}\label{Thm-L2}
	Under the standing assumptions,
	\begin{equation}\label{L2-est}
		\norm{T_\lambda f}_{L^2(\R^2)}\le C\,\lambda^{-\frac12\min\left\{\frac{1}{\hP},\ \frac1k\right\}}\bigl(\log(2+\lambda)\bigr)^{M_{P,k}}\norm{f}_{L^2(\R)}.
	\end{equation}
	The exponent $\tfrac12\min\{1/\hP,1/k\}$ is sharp.
\end{theorem}

The two terms in the minimum correspond to two extremal examples. Both operate on the same principle: restricting the domain of integration to keep the phase $\lambda P(x,y)t^k$ small, thereby preventing oscillation. When $\hP > k$, this is achieved by restricting the values of $P(x,y)$ to a sublevel set, yielding a Knapp-type example. When $\hP < k$, we instead restrict the variable $t$, producing a focusing example. Both constructions are detailed in Section~\ref{sec-L2}.

\begin{remark}
	For the special case $P(x,y)=xy$ and $k=1$, our theorem yields the sharp estimate
	\[
	\lambda^{-\frac12}\sqrt{\log(2+\lambda)},
	\]
	which improves upon the bound
	\[
	\lambda^{-\frac12}\log(2+\lambda)
	\]
	obtained in \cite{Tang}. We also note that the example demonstrating the sharpness of this estimate was already constructed in \cite{Tang}.
\end{remark}

Our second result uses complex interpolation for a general $P$, as explained in point (b) above.

\begin{theorem}\label{Thm-Lp}
	Under the standing assumptions, then
	\begin{equation}\label{Lp-est}
		\norm{T_\lambda f}_{L^{2k+2}(\R^2)}\le C\,\lambda^{-\frac{1}{2(k+1)\hP}}\bigl(\log(2+\lambda)\bigr)^{M_{P,1}/(k+1)}\norm{f}_{L^2(\R)}.
	\end{equation}
	 Moreover, the decay rate is sharp.
\end{theorem}

We can interpolate Theorem~\ref{Thm-L2} against the trivial $L^2\to L^\infty$ bound. This gives the full range of exponents when $\hP\ge k$.

\begin{corollary}\label{Cor-range}
	Under the standing assumptions, suppose $\hP\ge k$. Then for every $2\le p\le\infty$,
	\begin{equation}\label{range-est}
		\norm{T_\lambda f}_{L^{p}(\R^2)}\le C\,\lambda^{-\frac{1}{p\hP}}\bigl(\log(2+\lambda)\bigr)^{2M_{P,k}/p}\norm{f}_{L^2(\R)}.
	\end{equation}
	The exponent $\tfrac{1}{p\hP}$ is sharp for every such $p$.
\end{corollary}


\begin{remark}\label{Rem-monomial}
	Let $P(x,y)=x^m+y^n$ with $m,n\ge2$. The Newton diagram has a single compact edge joining $(m,0)$ and $(0,n)$. The bisectrix meets its interior at $\bigl(\tfrac{mn}{m+n},\tfrac{mn}{m+n}\bigr)$. So $\hP=\tfrac{mn}{m+n}=\bigl(\tfrac1m+\tfrac1n\bigr)^{-1}$ and $\nuP=0$. Then $\frac{1/}{\hP}=\tfrac1m+\tfrac1n$. In this case, Theorems~\ref{Thm-L2}--\ref{Thm-Lp} give the sharp balanced-case estimates for \eqref{old-O.I.O}. So our results cover the balanced case left open in \cite{Xu25}. The exponent $\tfrac1m+\tfrac1n$ in that paper is the reciprocal Newton height of the merged spatial phase.
\end{remark}

\section{The sharp \texorpdfstring{$L^2\to L^2$}{L2 to L2} estimate}\label{sec-L2}

\subsection{Reduction by \texorpdfstring{$TT^{*}$}{TT*} and the kernel estimate}
The operator $T_\lambda^{*}T_\lambda$ acts on $L^2(\R)$. Its kernel is
\begin{equation}\label{kernel-def}
	K(t,s)=\int_{\R^2}e^{i\lambda P(x,y)\left(t^k-s^k\right)}\,\overline{\psi(x,y,s)}\,\psi(x,y,t)\dint x\dint y.
\end{equation}
We get this by interchanging the order of integration in
\[
	T_\lambda^{*}T_\lambda f(t)=\int_\R K(t,s)f(s)\dint s.
\]
We know $\norm{T_\lambda}_{L^2\to L^2}^2=\norm{T_\lambda^{*}T_\lambda}_{L^2\to L^2}$. So we only need to bound \eqref{kernel-def}. The key point is that $W:=t^k-s^k$ is constant during the $(x,y)$ integration. Thus \eqref{kernel-def} is a scalar oscillatory integral with phase $P$ and large parameter $\lambda W$. We apply Varchenko's theorem.

\begin{lemma}[Varchenko's estimate \cite{Var76}]\label{Lem-Var}
	Let $Q$ be real-analytic near the origin in $\R^2$, with $Q(0)=0$ and a critical point at the origin. Let $\phi\in C_c^\infty$ be supported sufficiently close to the origin. In adapted coordinates, let $h_Q$ be the Newton height and $\nu_Q\in\{0,1\}$ the multiplicity. Then
	\begin{equation}\label{Var-est}
		\abs{\int_{\R^2}e^{i\mu Q(x,y)}\phi(x,y)\dint x\dint y}\le C\,(1+\abs{\mu})^{-1/h_Q}\bigl(\log(2+\abs{\mu})\bigr)^{\nu_Q}.
	\end{equation}
	Also the sublevel set estimate follows
	\begin{equation}\label{Sub-est}
		c\lambda^{-1/h_Q}\le \abs{\{(x,y)\in B(0,1) : \abs{Q(x,y)}\le 1/\lambda\}}\le C\lambda^{-1/h_Q}\log^{\nu_Q}(2+\lambda).
	\end{equation}
\end{lemma}

\begin{lemma}[kernel bound]\label{Lem-kernel}
	For all $s,t$,
	\begin{equation}\label{kernel-est}
		\abs{K(t,s)}\le C\min\left\{1,\ \bigl(\lambda\abs{t^k-s^k}\bigr)^{-1/\hP}\bigl(\log(2+\lambda\abs{t^k-s^k})\bigr)^{\nuP}\right\}.
	\end{equation}
\end{lemma}

\begin{proof}
	The trivial bound is
	\[
		\abs{K(t,s)}\le\norm{\psi}_\infty^2\,\abs{\supp\psi}\lesssim1.
	\]
	For the oscillatory bound, we fix $s,t$ and apply Lemma~\ref{Lem-Var}. We use $Q=P$, $\mu=\lambda(t^k-s^k)$, and amplitude $\phi(x,y)=\overline{\psi(x,y,s)}\psi(x,y,t)$. As $s,t$ vary, $\phi$ stays in a bounded subset of $C_c^\infty$. This gives the second term in \eqref{kernel-est}. Combining this with the trivial bound gives the minimum.
\end{proof}

\begin{remark}
	Because $W=t^k-s^k$ is constant during the $(x,y)$ integration, we can use the two-dimensional Varchenko estimate directly. This is a feature of the single merged oscillation in \eqref{Key-O.I.O}. It replaces the use of van der Corput lemma in \cite{Xu25}.
\end{remark}

\subsection{The Schur test}
Let $c\le1$ be the radius of the support of $\psi$ in the variable $t$. Then $K(t,s)=0$ if $\abs{s}>c$ or $\abs{t}>c$. By \eqref{kernel-est}, the kernel is symmetric in size: $\abs{K(t,s)}=\abs{K(s,t)}$. So Schur's test gives
\begin{equation}\label{schur}
	\norm{T_\lambda^{*}T_\lambda}_{L^2\to L^2}\le\sup_{\abs{t}\le c}\int_{\abs{s}\le c}\abs{K(t,s)}\dint s\le\sup_{\abs{t}\le c}\,J(t),
\end{equation}
where
\begin{equation}\label{Jdef}
	J(t):=\int_{\abs{s}\le c}\min\left\{1,\bigl(\lambda\abs{t^k-s^k}\bigr)^{-1/\hP}\bigl(\log(2+\lambda\abs{t^k-s^k})\bigr)^{\nuP}\right\}\dint s.
\end{equation}

\begin{lemma}\label{Lem-schur}
	Let $a:=1/\hP$ and let $M_P$ be as in \eqref{def-M}. Then
	\begin{equation}
		\sup_{\abs{t}\le c}J(t)\approx\lambda^{-\min\{a,\,1/k\}}\bigl(\log(2+\lambda)\bigr)^{2M_P}.
	\end{equation}
\end{lemma}

\begin{proof}
	First, we ignore the logarithmic factor $\bigl(\log(2+\lambda\abs{t^k-s^k})\bigr)^{\nuP}$ and estimate the pure power. We will add the logarithm back at the end. We split $J(t)=J_{\mathrm{near}}(t)+J_{\mathrm{far}}(t)$. The near part is over the set $E_{\mathrm{near}}=\{(t,s): \abs{t^k-s^k}\le1/\lambda\}$. The far part is over the complement.

	On $E_{\mathrm{near}}$, the integrand is at most $1$. We know
	\[
		\abs{\{s:\abs{s^k-t^k}\le\varepsilon\}}\lesssim\varepsilon^{1/k}
	\]
	uniformly in $t$. So
	\begin{equation}\label{near}
		\sup_t J_{\mathrm{near}}(t)\lesssim\lambda^{-1/k}.
	\end{equation}

	For $J_{\mathrm{far}}$, we check three ranges of $t$.

	\emph{(i) $\abs{t}\approx1$.} Near the singularity $s=t$, we have $\abs{t^k-s^k}\approx\abs{s-t}$. So
	\begin{equation}
		J_{\mathrm{far}}(t)\lesssim\lambda^{-a}\int_{1/\lambda}^{c}r^{-a}\dint r\approx\lambda^{-\min\{a,1\}}.
	\end{equation}
	(If $a=1$, there is a logarithm). If $k$ is even, the singularity at $s=-t$ gives the same bound.

	\emph{(ii) $\abs{t}\lesssim\lambda^{-1/k}$.} If $\abs{s}\gg\lambda^{-1/k}$, we have $\abs{t^k-s^k}\approx\abs{s}^k$. So
	\begin{equation}
		J_{\mathrm{far}}(t)\lesssim\lambda^{-a}\int_{\lambda^{-1/k}}^{c}s^{-ka}\dint s\approx
		\begin{cases}
			\lambda^{-1/k}, & ka>1,\\
			\lambda^{-1/k}\log\lambda, & ka=1,\\
			\lambda^{-a}, & ka<1.
		\end{cases}
	\end{equation}

	\emph{(iii) $\lambda^{-1/k}\ll\abs{t}\ll1$ (intermediate scale).}
	Let $a=1/\hP$. Since $\hP\ge1$, we have $0<a\le1$. Let $\mu:=\lambda k t^{k-1}$.
	We bound the integral \eqref{Jdef} over $E_{\mathrm{far}}$ by splitting $s\in[-c,c]$ into three zones. We write $r:=s-t$. Since $\lambda^{-1/k}\ll t\ll1$, we have $\lambda t^{k}\gg1$ (so $1/\mu\ll t$) and $\mu\approx\lambda t^{k-1}\le\lambda$.
	
	\emph{Zone A: the core $\abs{r}\le t/2$.}
	Here $s$ and $t$ have the same sign and $s\approx t$. By the mean value theorem, $\abs{t^k-s^k}=k\,\xi^{k-1}\abs{r}$ for some $\xi\approx t$. So
	\[
		\abs{t^k-s^k}\approx\mu\abs{r}/\lambda.
	\]
	Since $1/\mu\ll t$ and $0<a\le1$,
	\begin{equation}\label{iii-A}
		\int_{\{\abs{r}\le t/2\}\cap E_{\mathrm{far}}}\bigl(\lambda\abs{t^k-s^k}\bigr)^{-a}\dint s
		\approx\int_{1/\mu}^{t/2}(\mu r)^{-a}\dint r
		\approx\mu^{-a}t^{\,1-a}.
	\end{equation}
	The upper endpoint dominates because $t\gg1/\mu$ and $1-a\ge0$. (If $a=1$, the bound is $\mu^{-1}\log(\mu t)\lesssim\mu^{-1}\log\lambda$, which we absorb later.)
	
	\emph{Zone B: the collar $t/2<\abs{r}\le 3\abs{t}$.}
	Here $\abs{s}\lesssim\abs{t}$, but $s$ is far from $t$. So
	\[
		\abs{t^k-s^k}\gtrsim t^{k}.
	\]
	The zone has length $\lesssim\abs{t}$. So
	\begin{equation}\label{iii-B}
		\int_{\{t/2<\abs{r}\le3\abs{t}\}}\bigl(\lambda\abs{t^k-s^k}\bigr)^{-a}\dint s
		\lesssim\abs{t}\,\bigl(\lambda t^{k}\bigr)^{-a}
		=\lambda^{-a}\,t^{\,1-ak}
		\approx\mu^{-a}t^{\,1-a}.
	\end{equation}
	This matches Zone~A.
	
	\emph{Zone C: the far range $3\abs{t}<\abs{s}\le c$.}
	Here $\abs{s}\ge3\abs{t}$, so
	\[
		\abs{t^k-s^k}\approx\abs{s}^{k}.
	\]
	Therefore
	\begin{equation}\label{iii-C}
		\int_{3\abs{t}<\abs{s}\le c}\bigl(\lambda\abs{t^k-s^k}\bigr)^{-a}\dint s
		\approx\lambda^{-a}\int_{\abs{t}}^{c}s^{-ak}\dint s
		\lesssim
		\begin{cases}
			\lambda^{-a}, & ak<1,\\[1mm]
			\lambda^{-a}\log(1/\abs{t}), & ak=1,\\[1mm]
			\mu^{-a}t^{\,1-a}, & ak>1.
		\end{cases}
	\end{equation}
	If $ak>1$, the lower endpoint $s\approx\abs{t}$ dominates, giving $\mu^{-a}t^{1-a}$. If $ak<1$, the upper endpoint $s\approx c\approx1$ dominates, giving $\lambda^{-a}$.
	
	\emph{Conclusion of (iii).} Adding \eqref{iii-A}--\eqref{iii-C} gives
	\begin{equation}\label{iii-total}
		J_{\mathrm{far}}(t)\lesssim
		\begin{cases}
			\lambda^{-a}, & a<1/k,\\[1mm]
			\lambda^{-a}t^{\,1-ak}, & a>1/k,
		\end{cases}
	\end{equation}
	up to a factor $\log(1/\abs{t})\le\log\lambda$ at $a=1/k$.
	The term $\lambda^{-a}t^{\,1-ak}$ is a power of $t$. It is monotone on $t\in[\lambda^{-1/k},1]$. If $a<1/k$, it increases, and the supremum is $\lambda^{-a}$ at $t\approx1$. If $a>1/k$, it decreases, and the supremum is $\lambda^{-1/k}$ at $t\approx\lambda^{-1/k}$. In both cases, the supremum is already covered by cases (i) and (ii). So the intermediate scale does not change the maximum.

	Collecting \eqref{near} and (i)--(iii) and using $1/k\le1$,
	\begin{equation}
		\sup_{\abs{t}\le c}J(t)\approx\lambda^{-\min\{a,1/k\}}\qquad(a\neq1/k),
	\end{equation}
	with an extra $\log(2+\lambda)$ at $a=1/k$. Finally we reinstate the factor $\bigl(\log(2+\lambda\abs{t^k-s^k})\bigr)^{\nuP}$. On the dominant part of each integral, one has $\lambda\abs{t^k-s^k}\approx\lambda$. So this factor contributes $\bigl(\log(2+\lambda)\bigr)^{\nuP}$ when $a<1/k$. It is $O(1)$ when $a>1/k$. At $a=1/k$ it upgrades the single logarithm to $\bigl(\log(2+\lambda)\bigr)^{\nuP+1}$. These are exactly the powers $2M_{P,k}$ in \eqref{def-M}.
\end{proof}

\begin{proof}[Proof of the estimate in Theorem~\ref{Thm-L2}]
	By \eqref{schur} and Lemma~\ref{Lem-schur},
	\begin{equation}
		\norm{T_\lambda}_{L^2\to L^2}^2=\norm{T_\lambda^{*}T_\lambda}_{L^2\to L^2}\lesssim\lambda^{-\min\{1/\hP,1/k\}}\bigl(\log(2+\lambda)\bigr)^{2M_{P,k}},
	\end{equation}
	and taking square roots gives \eqref{L2-est}.
\end{proof}

\subsection{Optimality}
Following \cite{Xu23,Xu25}, we take a nonnegative cut-off with
\begin{equation}\label{Supp-Func}
	\psi(x,y,t)=\begin{cases}
		0, &\quad\abs{(x,y,t)}\ge c,\\
		1, &\quad\abs{(x,y,t)}\le\frac{c}{2},
	\end{cases}
\end{equation}
where $c$ is a small positive number.

Assume the a priori estimate
\[
	\norm{T_\lambda f}_{L^2(\R^2)}\le C_\psi\lambda^{-\delta}\norm{f}_{L^2(\R)}.
\]

\emph{Knapp example (spatial degeneracy).} Let $f(t)=\chi_{[0,1]}(t)$, so $\norm{f}_{L^2}=1$. If $\abs{(x,y)}\le\tfrac12$ and $\lambda\abs{P(x,y)}\le1$, then $\lambda\abs{P(x,y)t^k}\le1$ because $\abs{t}\le1$. So there is no oscillation and $\abs{T_\lambda f(x,y)}\gtrsim1$. By the lower sublevel bound in Lemma~\ref{Lem-Var}, this set of $(x,y)$ has measure $\gtrsim\lambda^{-1/\hP}$. Thus
\begin{equation}
	\lambda^{-\frac{1}{2\hP}}\lesssim\Bigl[\int_{\{\lambda\abs{P}\le1\}}\abs{T_\lambda f(x,y)}^2\dint x\dint y\Bigr]^{1/2}\le C_\psi\lambda^{-\delta}.
\end{equation}
This implies $\delta\le\tfrac{1}{2\hP}$.

\emph{Focusing example (temporal degeneracy).} Since $P$ is continuous and $P(0,0)=0$, there is a ball $B_0$ of positive radius where $\abs{P}\le1$. Let $f=\chi_{[0,\rho]}$ with $\rho=\lambda^{-1/k}$, so $\norm{f}_{L^2}=\lambda^{-1/(2k)}$. For $(x,y)\in B_0$ and $0\le t\le\rho$, we have $\lambda\abs{P(x,y)t^k}\le\lambda\cdot1\cdot\rho^{k}=1$. So $\abs{T_\lambda f(x,y)}\gtrsim\rho=\lambda^{-1/k}$ on $B_0$. Therefore $\norm{T_\lambda f}_{L^2}\gtrsim\lambda^{-1/k}$, and
\begin{equation}
	\lambda^{-1/k}\lesssim C_\psi\lambda^{-\delta}\lambda^{-1/(2k)}.
\end{equation}
This forces
\[
	\delta\le\tfrac1{2k}.
\]

Together, these two examples give $\delta\le\tfrac12\min\{1/\hP,1/k\}$. This matches the exponent in \eqref{L2-est}. This completes the proof of Theorem~\ref{Thm-L2}. \qed


\section{The complex interpolation and \texorpdfstring{$L^2\to L^p$}{L2 to Lp} estimates}\label{sec-Lp}

We now prove Theorem~\ref{Thm-Lp} following the strategy of \cite{Xu25}.

\subsection{Reduction}
We rewrite \eqref{Lp-est} in dual form. Let $p=2k+2$ and $p'=\tfrac{2k+2}{2k+1}$. The dual operator is
\begin{equation}
	T_\lambda^{*}g(t)=\int_{\R^2}e^{-i\lambda P(u,v)t^k}\overline{\psi}(u,v,t)g(u,v)\dint u\dint v.
\end{equation}
Exactly as in \cite{Xu25}, we have
\begin{equation}\label{Dual-Red}
	\begin{split}
	\int_{\R}\abs{T_\lambda^{*}g(t)}^2\dint t&=\int_{\R^2}T_Kg(x,y)\,\overline{g}(x,y)\dint x\dint y,\\
	T_Kg(x,y)&=\int_{\R^2}K(x,y,u,v)g(u,v)\dint u\dint v,
	\end{split}
\end{equation}
where
\begin{equation}\label{TK-kernel}
	K(x,y,u,v)=\int_{\R}e^{i\lambda\left[P(x,y)-P(u,v)\right]t^k}\overline{\psi}(u,v,t)\psi(x,y,t)\dint t .
\end{equation}
Thus \eqref{Lp-est} is equivalent to the bound
\begin{equation}\label{TK-target}
	T_K:L^{p'}(\R^2)\to L^{p}(\R^2)\quad\text{with norm}\quad\lesssim\lambda^{-\frac{1}{(k+1)\hP}}\bigl(\log(2+\lambda)\bigr)^{2M_{P,1}/(k+1)}.
\end{equation}
Let
\begin{equation}\label{Phi-W}
	W:=P(x,y)-P(u,v).
\end{equation}
Then $K$ depends on $(x,y,u,v)$ through the scalar $W$ in the phase.

\subsection{The merged analytic family}
Following \cite[Ch.~IX]{Ste93}, we fix a smooth cutoff $\zeta\in C_c^\infty$ with $\zeta(s)=1$ for $\abs{s}\le1$. Let
\begin{equation}\label{delta-alpha}
	\delta_\alpha(s)=\begin{cases}
		\dfrac{e^{\alpha^2}}{\Gamma(\alpha)}s^{\alpha-1}\zeta^2(s), & s>0,\\[2mm]
		0, & s\le0.
	\end{cases}
\end{equation}
This is an analytic continuation of the Dirac mass, with $\delta_0=\delta$. We embed $T_K$ into the analytic family
\begin{equation}\label{TKalpha}
	T_K^{\alpha}g(x,y)=\int_{\R^2}K^{\alpha}(x,y,u,v)g(u,v)\dint u\dint v,
\end{equation}
where
\begin{equation}\label{Kalpha}
	K^{\alpha}(x,y,u,v)=\int_{\R}\int_{\R}e^{i\lambda W\left(t^k+s\right)}\overline{\psi}(u,v,t)\psi(x,y,t)\,\delta_\alpha(s)\dint t\dint s .
\end{equation}
Since $\delta_0=\delta$, we have $K^{0}=K$ and $T_K^{0}=T_K$. We now bound the two endpoints.

\begin{proposition}[$L^1\to L^\infty$ endpoint]\label{Prop-Linfty}
	For $\mathrm{Re}\,\alpha=-1/k$,
	\begin{equation}\label{Linfty-est}
		\norm{T_K^{\alpha}g}_{L^\infty(\R^2)}\le C\norm{g}_{L^1(\R^2)}.
	\end{equation}
\end{proposition}

\begin{proof}
	Carrying out the $s$-integration in \eqref{Kalpha} and writing
	\begin{equation}\label{Kzero}
		K^{(0)}(x,y,u,v)=\int_\R e^{i\lambda Wt^k}\overline{\psi}(u,v,t)\psi(x,y,t)\dint t
	\end{equation}
	for the ($\alpha$-independent) inner factor, we obtain
	\begin{equation}\label{split}
		K^{\alpha}(x,y,u,v)=K^{(0)}(x,y,u,v)\cdot\widehat{\delta_\alpha}\!\left(\frac{\lambda W}{2\pi}\right).
	\end{equation}
	By van der Corput's lemma of order $k$ applied to the merged phase $\lambda Wt^k$,
	\begin{equation}\label{vdc-k}
		\abs{K^{(0)}(x,y,u,v)}\le C\bigl(1+\lambda\abs{W}\bigr)^{-1/k}.
	\end{equation}
	The standard bound for the Fourier transform of $\delta_\alpha$ (see \cite[Ch.~IX]{Ste93}), $\abs{\widehat{\delta_\alpha}(\tau)}\le C(1+\abs{\tau})^{-\mathrm{Re}\,\alpha}$, gives at $\mathrm{Re}\,\alpha=-1/k$
	\begin{equation}\label{delta-hat}
		\Bigl|\widehat{\delta_\alpha}\!\left(\tfrac{\lambda W}{2\pi}\right)\Bigr|\le C\bigl(1+\lambda\abs{W}\bigr)^{1/k}.
	\end{equation}
	Since \eqref{vdc-k} and \eqref{delta-hat} are expressed through the \emph{same} merged variable $W$, their product is uniformly bounded, and $\norm{T_K^\alpha g}_{L^\infty}\le\sup\abs{K^\alpha}\,\norm{g}_{L^1}\le C\norm{g}_{L^1}$.
\end{proof}

\begin{proposition}[$L^2\to L^2$ endpoint]\label{Prop-L2}
	Assume $\mathrm{Re}\,\alpha=1$. We have the bound
	\begin{equation}\label{L2endpoint}
		\norm{T_K^{\alpha}g}_{L^2(\R^2)}\le C(\alpha)\,\lambda^{-1/\hP}\bigl(\log(2+\lambda)\bigr)^{2M_{P,1}}\norm{g}_{L^2(\R^2)},
	\end{equation}
	 where $C(\alpha)$ grows at most subexponentially in $\mathrm{Im}\,\alpha$.
\end{proposition}
\begin{proof}
	We write $T_K^\alpha$ as a composition of two operators. We have
	\begin{align}
		T_K^\alpha g(x,y)
		&=\int_{\R^2}K^\alpha(u,v,x,y) \bar{g}(u,v)\dint u\dint v\\
		&=\int_{\R^2}\int_{\R^2} e^{i\lambda\left[P(x,y)-P(u,v)\right](t^k+s)}\bar{\psi}(u,v,t)\psi(x,y,t)\delta_\alpha(s)\dint t\dint s\, g(u,v)\dint u\dint v\\
		&:=S_1\circ S_2(g)(x,y),
	\end{align}
	where
	\begin{equation}
		S_1(h)(x,y)=\frac{e^{\alpha^2}}{\Gamma(\alpha)}\int_0^{+\infty}\int_\R e^{i\lambda P(x,y)(t^k+s)}\psi(x,y,t)\zeta(s) s^{\alpha-1}h(t,s)\dint t\dint s,
	\end{equation}
	and 
		\begin{equation}
		S_2(g)(t,s)=\int_{\R^2} e^{-i\lambda P(u,v)(t^k+s)}\bar{\psi}(u,v,t)\zeta(s) g(u,v)\dint u\dint v.
	\end{equation}
	Applying the coordinate transformation
	\begin{equation}
		w=s+t^k, t=t,
	\end{equation}
	we rewrite $S_1(h)(x,y)$ as
		\begin{equation}
		S_1(h)(x,y)=\frac{e^{\alpha^2}}{\Gamma(\alpha)}\int_{t^k}^{+\infty}\int_\R e^{i\lambda P(x,y)w}\psi(x,y,t)\zeta(w-t^k) (w-t^k)^{\alpha-1}h(t,w-t^k)\dint t\dint w.
	\end{equation}
	We set 
	\begin{equation}
		\tilde{\psi}(x,y,w,t)=\psi(x,y,t)\zeta(w-t^k),\quad \tilde{h}(w,t)=(w-t^k)^{\alpha-1}h(t,w-t^k)\chi_{[0,+\infty)}(w-t^k).
	\end{equation}
	It follows that
	\begin{equation}
		S_1(h)(x,y)=\frac{e^{\alpha^2}}{\Gamma(\alpha)}\int_{\R^2} e^{i\lambda P(x,y)w}\tilde{\psi}(x,y,w,t) \tilde{h}(w,t)\dint t\dint w.
	\end{equation}
	In the same way, we express $S_2(g)$ in the new coordinates $(w,t)$ as
	\begin{equation*}
		S_2(g)(w,t)=\int_{\R^2} e^{-i\lambda P(u,v)w}\tilde{\psi}(u,v,w,t) g(u,v)\dint u\dint v,
	\end{equation*}
	where
	\begin{equation*}
		\tilde{\psi}(u,v,w,t)=\bar{\psi}(u,v,w-t^k)\zeta(w-t^k).
	\end{equation*}
	The phase in $w$ has degree one. So we can apply Theorem~\ref{Thm-L2} twice, in the case $k=1$. This gives
	\begin{align}
		\norm{S_1(S_2(g))}_{L^2(\R^2)}^2
		&=\int_{\R^2}\abs{\frac{e^{\alpha^2}}{\Gamma(\alpha)}\int_{\R^2}e^{i\lambda P(x,y)w}\tilde{\psi}(x,y,w,t) S_2(g)(w,t)\dint w\dint t}^2\dint x\dint y\\
		&\leq C(\alpha)\int_{\R}\int_{\R}\abs{\int_{\R^2}e^{i\lambda P(x,y)w}\tilde{\psi}(x,y,w,t) S_2(g)(w,t)\dint w\dint t}^2\dint x\dint y\\
		&\leq \int_{\R}C(\alpha,t)\int_{\R^2}\abs{\int_{\R} e^{i\lambda P(x,y)w}\tilde{\psi}(x,y,w,t) S_2(g)(w,t)\dint w}^2\dint x\dint y\dint t\\
		&\leq \lambda^{-\frac{1}{\hP}}\bigl(\log(2+\lambda)\bigr)^{2M_{P,1}}\int_{\R}\int_{\R}C(\alpha,t)\abs{S_2(g)(w,t)}^2\dint w\dint t\\
		&=\lambda^{-\frac{1}{\hP}}\bigl(\log(2+\lambda)\bigr)^{2M_{P,1}}\\
		&\qquad\times\int_{\R}\int_{\R}C(\alpha,t)\abs{\int_{\R^2} e^{-i\lambda P(u,v)w}\tilde{\psi}(u,v,w,t) g(u,v)\dint u\dint v}^2\dint w\dint t\\
		&\leq C(\alpha) \lambda^{-\frac{2}{\hP}}\bigl(\log(2+\lambda)\bigr)^{4M_{P,1}}\norm{g}_{L^2(\R^2)}^2.
	\end{align}
\end{proof}

Proposition~\ref{Prop-L2} is the key step. It explains why a general $P$ is allowed. In \cite{Xu25}, the substitution $w=s+t^k$ was used to separate the two monomials into $x^m w+y^n t^l$. This allowed the one-dimensional Phong--Stein estimates to be iterated. That separation does not work for a general $P$. Here, the same substitution merges $t^k$ and $s$ into a single degree-one phase $\lambda W w$. No separation is needed. The endpoint is exactly the $k=1$ case of the merged Varchenko/Schur argument from Section~\ref{sec-L2}. It is governed only by the two-dimensional Newton height $\hP$. So the interpolation works for every $P$. The exponent is $1/\hP$ instead of $\min\{1/\hP,1/k\}$. In the degree-one case, the focusing regime vanishes entirely, as $1/\hP\le1$.

\subsection{Interpolation}
The family $\{T_K^{\alpha}\}$ is analytic and has admissible growth on the strip $-1/k\le\mathrm{Re}\,\alpha\le1$. We apply Stein's complex interpolation to Propositions~\ref{Prop-Linfty} and~\ref{Prop-L2}. At $\alpha=0$, we get the parameter
\begin{equation}
	0=(1-\theta)\cdot1+\theta\cdot\bigl(-\tfrac1k\bigr)\ \Longrightarrow\ \theta=\frac{k}{k+1},\qquad 1-\theta=\frac{1}{k+1}.
\end{equation}
The exponents are
\begin{equation}
	\frac1p=\frac{1-\theta}{2}+\theta\cdot0=\frac{1}{2(k+1)},\qquad p=2k+2,\quad p'=\tfrac{2k+2}{2k+1}.
\end{equation}
We interpolate the endpoint decay $\lambda^{-1/\hP}(\log(2+\lambda))^{2M_{P,1}}$ from \eqref{L2endpoint} (weight $1-\theta$) against $\lambda^{0}$ (weight $\theta$). This gives
\begin{equation}
	\norm{T_K}_{L^{p'}\to L^{p}}\le C\,\lambda^{-\frac{1}{(k+1)\hP}}\bigl(\log(2+\lambda)\bigr)^{2M_{P,1}/(k+1)},\qquad p=2k+2 .
\end{equation}
By \eqref{Dual-Red}, this means
\[
	\norm{T_\lambda^{*}g}_{L^2}\le C\lambda^{-\frac{1}{2(k+1)\hP}}\bigl(\log(2+\lambda)\bigr)^{M_{P,1}/(k+1)}\norm{g}_{L^{p'}}.
\]
Dualizing this estimate yields \eqref{Lp-est}, which concludes the proof of Theorem~\ref{Thm-Lp}.

\subsection{Optimality}
We reuse the two examples of Section~\ref{sec-L2}, now with a general exponent $p$, and the a priori estimate $\norm{T_\lambda f}_{L^p}\le C_\psi\lambda^{-\delta}\norm{f}_{L^2}$.

\emph{Knapp}, with $f=\chi_{[0,1]}$: on the set $\{(x,y):\lambda\abs{P(x,y)}\le1\}$ we have $\abs{T_\lambda f}\gtrsim1$. This set has measure $\gtrsim\lambda^{-1/\hP}$. So
\begin{equation}
	\lambda^{-\frac{1}{p\hP}}\lesssim\Bigl[\int_{\{(x,y):\,\lambda\abs{P(x,y)}\le1\}}\abs{T_\lambda f(x,y)}^p\dint x\dint y\Bigr]^{1/p}\le C_\psi\lambda^{-\delta}.
\end{equation}
This implies $\delta\le\tfrac{1}{p\hP}$.

\emph{Focusing}, with $f=\chi_{[0,\lambda^{-1/k}]}$: as in Section~\ref{sec-L2}, $\norm{T_\lambda f}_{L^p}\gtrsim\lambda^{-1/k}$ while $\norm{f}_{L^2}=\lambda^{-1/(2k)}$, so
\[
	\delta\le\tfrac1{2k}.
\]

Thus for every $p$ the sharp exponent is at most $\min\{\tfrac{1}{p\hP},\tfrac1{2k}\}$. At $p=2k+2$ this equals
\[
	\min\bigl\{\tfrac{1}{2(k+1)\hP},\tfrac1{2k}\bigr\}.
\]
Since $\hP\ge1$ forces $\tfrac{1}{2(k+1)\hP}\le\tfrac{1}{2k}$, the minimum is $\tfrac{1}{2(k+1)\hP}$ for every $\hP$, matching the power in \eqref{Lp-est}. So the decay exponent in \eqref{Lp-est} is sharp for every $\hP$. This completes the proof of Theorem~\ref{Thm-Lp}. \qed

\begin{proof}[Proof of Corollary~\ref{Cor-range}]
	Assume $\hP\ge k$, so $1/\hP\le1/k$. Theorem~\ref{Thm-L2} gives the endpoint $p=2$ with rate $\tfrac{1}{2\hP}$ and log power $M_{P,k}$. The trivial bound is
	\[
		\norm{T_\lambda f}_{L^\infty}\le\norm{\psi}_\infty\norm{f}_{L^1}.
	\]
	Using Cauchy--Schwarz in $t$ over the unit support, this gives the endpoint $p=\infty$ with rate $0$. We interpolate these two bounds using the Riesz--Thorin theorem. For $\tfrac1p=\tfrac{1-\theta}{2}$, we get
	\begin{equation}
		\norm{T_\lambda f}_{L^p}\lesssim\lambda^{-\frac{1-\theta}{2}\cdot\frac{1}{\hP}}\bigl(\log(2+\lambda)\bigr)^{2M_{P,k}/p}\norm{f}_{L^2}=\lambda^{-\frac{1}{p\hP}}\bigl(\log(2+\lambda)\bigr)^{2M_{P,k}/p}\norm{f}_{L^2}.
	\end{equation}
	Sharpness for each $p$ follows from the Knapp example. Since $\hP\ge k$, we have $\tfrac{1}{p\hP}\le\tfrac{1}{2\hP}\le\tfrac1{2k}$. Consequently, the focusing example plays no role, and the minimum simplifies to $\tfrac{1}{p\hP}$.
\end{proof}


\section*{Acknowledgments}
The author is supported by the National Natural Science Foundation of China (Grant No. 12501124) and the Jiangsu Natural Science Foundation (Grant No. BK20200308).

\end{document}